\documentclass{amsart}

\usepackage[english]{babel}
\usepackage{enumitem}
\usepackage{amssymb}

\theoremstyle{plain}
\newtheorem{theorem}{Theorem}[section]
\newtheorem{lemma}[theorem]{Lemma}
\newtheorem{proposition}[theorem]{Proposition}
\newtheorem{corollary}[theorem]{Corollary}
\newtheorem{conjecture}[theorem]{Conjecture}

\theoremstyle{remark}

\newtheorem{remark}[theorem]{Remark}

\newcommand{\Q}{\mathbb{Q}}

\DeclareMathOperator{\Id}{Id}
\DeclareMathOperator{\length}{\mu}

\begin{document} 

\title[Free algebras]{Free algebras in division rings with an involution}
\author{Vitor O. Ferreira}
\address{Department of Mathematics, University of S\~{a}o Paulo, S\~{a}o Paulo, 05508-090, Brazil}
\email{vofer@ime.usp.br}
\thanks{The first author was partially supported by Fapesp-Brazil, 
Proj.~Tem\'atico 2009/52665-0.}
\author{\'Erica Z. Fornaroli}
\address{Department of Mathematics, Universidade Estadual de Maring\'{a}, Paran\'{a}, 87020-900, Brazil}
\email{ezancanella@uem.br}
\author{Jairo Z. Gon\c{c}alves}
\address{Department of Mathematics, University of S\~{a}o Paulo, S\~{a}o Paulo, 05508-090, Brazil}
\email{jz.goncalves@usp.br}
\thanks{The third author was supported by Grant CNPq 300.128/2008-8 and by Fapesp-Brazil, Proj.~Tem\'atico 2009/52665-0}

\begin{abstract}
Some general criteria to produce explicit free algebras inside the division
ring of fractions of skew polynomial rings are presented. These criteria are applied
to some special cases of division rings with natural involutions, yielding, for instance, 
free subalgebras generated by symmetric elements both in the division ring of fractions of
the group algebra of a torsion free nilpotent group and in the division ring of
fractions of the first Weyl algebra. 
\end{abstract}

\maketitle

\section{Introduction}

It has been conjectured by Makar-Limanov in \cite{ML84C} that a division ring which is
infinite dimensional over its center $k$ and finitely generated (as a division algebra
over $k$) must contain a noncommutative free $k$-subalgebra. Makar-Limanov himself provided
evidence for this in \cite{ML83}, where it is proved that the division ring of fractions
of the first Weyl algebra over the rational numbers contains a free subalgebra of rank $2$,
and in \cite{ML84}, where the case of the division ring of fractions of a group algebra of
a torsion free nonabelian nilpotent groups is tackled. Various authors have
dealt with this problem and Makar-Limanov's conjecture has been verified
in many families of division rings (see, e.g., \cite{mL86, MM91, GS96, FGS96, SG98, aL99, SG99,
BR12, GS12, GT12, FGS13, BR14, jS14, FG15, FGS15, BG16}).

Division rings often come equipped with an involution. That is the case, for
instance, of division rings of fractions of group algebras which are Ore
domains. These have natural involutions induced by involutions on the group.

After the work in \cite{GS12}, it has become apparent that an involutional
version of Makar-Limanov's conjecture should be investigated. To be more precise,
given a field $k$ and a division $k$-algebra $D$, a $k$-linear map $\ast\colon D\to D$
satisfying $(ab)^{\ast}=b^{\ast}a^{\ast}$ and $a^{\ast\ast}=a$ for all
$a,b\in D$ is called a \emph{$k$-involution}. An element $a\in D$ is said to
be \emph{symmetric} with respect to the involution $\ast$ if $a^{\ast}=a$.
Our aim in this paper is to contribute with supporting evidence to the
following conjecture.

\begin{conjecture}\label{conj}
Let $D$ be a division ring with center $k$, and let $\ast$ be
a $k$-involution on $D$. If $D$ is infinite dimensional over $k$ and
finitely generated as a division $k$-algebra, then there exist
two symmetric elements in $D$ which freely generate a free $k$-subalgebra
of $D$.
\end{conjecture}

In \cite{FGS13}, this conjecture has been proved to hold for the division
ring of fractions, inside the division ring of Malcev-Neumann series,
of the group algebra of a nonabelian orderable group $G$ with respect to 
an involution induced by the canonical (inverting) involution
on $G$.

Here, we present proofs to the following two further special cases of
Conjecture~\ref{conj}, which can be regarded as involutional versions of 
Makar-Limanov's early results.

\begin{theorem}\label{th:nilpgp}
Let $D$ be the division ring of fractions of the group 
algebra $k\Gamma$ of the Heisenberg 
group $\Gamma$ over the field $k$ and let $\ast$ be a $k$-involution of 
$D$ which is induced from an involution on $\Gamma$. Then $D$ contains a free 
$k$-algebra of rank $2$ freely generated by symmetric elements.
\end{theorem}

By the Heisenberg group, one understands the free nilpotent group of class
$2$ generated by $2$ elements. It can be presented by
$$\Gamma = \langle{x,y : [[x,y],x]=[[x,y],y]=1}\rangle,$$
where $[g,h]$ denotes the commutator $g^{-1}h^{-1}gh$ of elements $g,h$
in a group.

\begin{theorem}\label{th:weyl}
Let $A_1=\Q\langle{s,t : st-ts=1}\rangle$
denote the first Weyl algebra over the field $\Q$ of rational numbers and
let $\ast$ denote the $\Q$-involution of $A_1$ such that $s^{\ast}=-s$ and
$t^{\ast}=t$. Then the division ring of fractions $D_1$ of $A_1$ contains a
free $\Q$-subalgebra of rank $2$ freely generated by symmetric elements with
respect to the extension of $\ast$ to $D_1$.
\end{theorem}

Theorems~\ref{th:nilpgp} and \ref{th:weyl} will follow from criteria
that generalize the method developed by Bell and Rogalski in \cite{BR12}.
These will also provide simpler proofs of \cite[Theorem~A]{SG98} and
\cite[Theorem~1]{SG99}. As a special case, we obtain the following result.

\begin{theorem}\label{th:ratfunct}
Let $F$ be a field, let $K=F(X_{1}, \dots, X_{n})$
be the rational function field in $n$ indeterminates over $F$, and let $\sigma$ 
be an $F$-automorphism of $K$ of infinite order that
extends one from the polynomial algebra $F[X_{1}, \dots, X_{n}]$. 
Then, the division algebra $K(X;\sigma)$ contains
a noncommutative free $F$-subalgebra.
\end{theorem}

\section{Free subalgebras of fields of fractions of skew polynomial rings}

In this section we offer generalizations of the method of \cite{BR12} to
construct free algebras inside division ring of fractions of skew polynomial rings.

Let $k$ be a field and let $D$ be a division $k$-algebra. Let $\sigma\colon
D\to D$ be a $k$-automorphism and let $\delta\colon D\to D$ be a 
$\sigma$-derivation (that is, a $k$-linear map satisfying
$\delta(\alpha\beta)=\sigma(\alpha)\delta(\beta)+\delta(\alpha)\beta$, for all 
$\alpha,\beta\in D$). Denote by $D[X;\sigma,\delta]$
the skew polynomial ring in the indeterminate $X$ such that $X\alpha=\sigma(\alpha)X+
\delta(\alpha)$, for
all $\alpha\in D$, and let $D(X;\sigma,\delta)$ denote its division ring of fractions. Given 
$a_0,a_1,b_0,b_1\in k$, consider the polynomials 
$f=a_0+a_1X, g=b_0+b_1X\in k[X]\subseteq
D[X;\sigma,\delta]$. Also, let $\psi\colon D\to D$ be the map defined by 
$\psi=a_1\delta+a_0(\Id-\sigma)$, where $\Id$ stands for the identity map 
from $D$ to $D$. (Note that
$\psi$ is again a $\sigma$-derivation.) Finally, let $E=\ker\psi$.

In what follows, we will further assume that $a_1\neq 0$ and that 
$\Xi=gf^{-1}\in D(X;\sigma,\delta)
\setminus k$.

Under these hypotheses, we shall prove the following two theorems.

\begin{theorem}\label{th:1}
Let $\alpha\in D$ be such that
\begin{itemize}
\item $\{1,\alpha,\alpha^2\}$ is left linearly independent over $\sigma(E)$ and
\item $\psi(D)\cap\bigl(\sigma(E)+\sigma(E)\alpha+\sigma(E)\alpha^2\bigr) = \{0\}.$
\end{itemize}
If either
\begin{enumerate}[label=(\roman*)]
\item $b_1=0$ or \label{cond:1}
\item $b_0=0$ and $\delta=0$, \label{cond:2}
\end{enumerate}
then the set $\{\alpha\Xi,\Xi\alpha\}$ freely generates a free $k$-subalgebra in $D(X;\sigma,\delta)$.
\end{theorem}

\begin{proof}
Consider the set
$$
S=\{(i_1,\dots,i_t) : t\geq 1, i_j\in\{0,1,2\}, \text{ for all }j\in\{1,\dots, t\}\}.
$$
Given $I=(i_1,\dots,i_t)\in S$, consider the elements in $D(X;\sigma,\delta)$ defined by
$$
R_I = \alpha^{i_1}\Xi\alpha^{i_2}\Xi\dots\alpha^{i_{t-1}}\Xi\alpha^{i_t}\Xi\alpha
$$
and
$$
L_I = \alpha^{i_1}\Xi\alpha^{i_2}\Xi\dots\alpha^{i_{t-1}}\Xi\alpha^{i_t}\alpha\Xi.
$$
The set $\mathcal{B}=\{1\}\cup \{R_I : I\in S\}\cup \{L_1 : I\in S\}$ (properly)
contains all the words in the letters $\alpha\Xi$ and $\Xi\alpha$. Therefore, if we prove that
$\mathcal{B}$ is linearly independent over $k$, we will have proved that $\alpha\Xi$ and
$\Xi\alpha$ freely generate a free $k$-algebra.

In order to show that $\mathcal{B}$ is indeed linearly independent over $k$,
we shall introduce new auxiliary elements. Given $I=(i_1,\dots,i_t)\in S$, let
$$
V_I = \Xi\alpha^{i_1}\Xi\alpha^{i_2}\Xi\dots\alpha^{i_{t-1}}\Xi\alpha^{i_t}\Xi\alpha,
$$
that is, $V_I=\Xi R_I$. We shall also define $V_{\varnothing}=\Xi$.

Given $I=(i_1,\dots,i_t)\in S$, define the \emph{truncation} of $I$
to be $I^{\prime}=(i_2,\dots,i_t)$ if $t\geq 2$, and $I^{\prime}=\varnothing$ if $t=1$.
So, in $D(X;\sigma,\delta)$, the following relations hold:
\begin{equation}\label{eq:vi}
\Xi^{-1}V_{\varnothing}=1
\quad\text{and}\quad
\Xi^{-1}V_I=R_I=\alpha^{i_1}V_{I^{\prime}},
\end{equation}
for all $I\in S$.

For $I=(i_1,\dots, i_t)\in S$, we define the \emph{length} of $I$ to be
$\length(I)=t$. Also, we set $\length(\varnothing)=0$.

We claim that if $\{V_I : I\in S\cup\{\varnothing\}\}$ is left linearly
independent over $D$, then $\mathcal{B}$ is linearly independent over $k$.
Indeed, suppose $\{V_I : I\in S\cup\{\varnothing\}\}$ is left linearly
independent over $D$ and that
\begin{equation}\label{eq:1}
b+\sum_{I\in S}c_IR_I + \sum_{I\in S}d_IL_I=0
\end{equation}
is a linear combination of elements of $\mathcal{B}$ with coefficients
$b, c_I, d_I$ from $k$ resulting in $0$. Multiplying \eqref{eq:1} by $\Xi\alpha$ on
the right, one obtains a relation of the form
\begin{equation}\label{eq:2}
\sum_{I\in S} e_IR_I=0,
\end{equation}
with $e_I\in k$. Note that, by doing that, all of the elements $R_I$ in \eqref{eq:2}
are distinct. Hence, in view of \eqref{eq:vi}, we get
$$
0=\sum_{I\in S}e_IR_I=\sum_{I\in S}e_I\alpha^{i_1}V_{I^{\prime}}.
$$
For each $I=(i_1,\dots,i_t)\in S$, there are exactly 3 elements in $S$ which have 
truncation $I^{\prime}$, they are
$$ 
I_0 = (0,i_1,\dots,i_t), \quad I_1=(1,i_1,\dots,i_t) \quad\text{and}\quad
I_2=(2,i_1,\dots,i_t).
$$
Thus, since $\{V_I : I\in S\cup\{\varnothing\}\}$ is left linearly
independent over $D$, it follows that, for each $I\in S$, one has
$$
e_{I_0} + e_{I_1}\alpha + e_{I_2}\alpha^2=0.
$$
But, by hypothesis, $\{1,\alpha,\alpha^2\}$ is linearly independent over $k$ (for
$\sigma(E)\supseteq k)$; therefore, $e_{I_0}=e_{I_1}=e_{I_2}=0$. This proves
that all the coefficients in \eqref{eq:2}, which are the same as the ones
in \eqref{eq:1}, are zero. So, $\mathcal{B}$ is linearly independent over $k$.

Our next task is to show that $\{V_I : I\in S\cup\{\varnothing\}\}$ is left linearly
independent over $D$. We shall split the proof in two parts, depending on the
conditions \ref{cond:1} or \ref{cond:2} in the statement of the theorem.

\medskip

\textit{First suppose that condition \ref{cond:1} holds, that is, that $b_1=0$.} 
In this case, we must have $b_0\ne 0$. We shall show the stronger statement 
that $\{V_I : I\in S\cup\{\varnothing\}\}$
is left linearly independent over $D$ modulo the subspace $D[X;\sigma,\delta]$.
By contradiction, suppose there exists a relation
\begin{equation}\label{eq:3}
\sum_{I\in S\cup\{\varnothing\}} \beta_IV_I = h\in D[X;\sigma,\delta],
\end{equation}
with $\beta_I\in D$ not all zero. Among all those relations, choose one with 
$r=\max\{\length(I) : \beta_I\neq 0\}$ minimal. Moreover, among those, choose
one with the smallest number of nonzero coefficients $\beta_I$ for $I$ with
$\length(I)=r$. Note that $r\geq 1$, otherwise we would have $\Xi\in D[X;\sigma,\delta]$, 
which is impossible. Clearly, we can further assume that our relation \eqref{eq:3},
beyond being minimal in the sense described above, has $\beta_T=1$ for some
$T\in S$ with $\length(T)=r$, by multiplying it by a nonzero element of $D$ on
the left if necessary.

Recall that $\Xi = gf^{-1}=b_0(a_0+a_1X)^{-1}$. Hence, $\Xi^{-1}=(a_0+a_1X)b_0^{-1}$.
It, then, follows from \eqref{eq:vi} that
\begin{equation}\label{eq:4}
XV_{\varnothing}=-a_1^{-1}a_0V_{\varnothing} + a_1^{-1}b_0
\quad\text{and}\quad
XV_I=-a_1^{-1}a_0V_I+a_1^{-1}b_0\alpha^{i_1}V_{I^{\prime}},
\end{equation}
for all $I\in S$. Multiplying \eqref{eq:3} by $X$ on the left, and using \eqref{eq:4}, yields
\begin{equation*}\begin{split}
Xh & = \sum_{I\in S\cup\{\varnothing\}} X\beta_IV_I =
\sum_{I\in S\cup\{\varnothing\}}\bigl(\sigma(\beta_I)X+\delta(\beta_I)\bigr)V_I\\
& = \sigma(\beta_{\varnothing})XV_{\varnothing} + \delta(\beta_{\varnothing})V_{\varnothing} +
\sum_{I\in S} \bigl(\sigma(\beta_I)X+\delta(\beta_I)\bigr)V_I\\
& =\sigma(\beta_{\varnothing})(-a_1^{-1}a_0V_{\varnothing} + a_1^{-1}b_0) + 
\delta(\beta_{\varnothing})V_{\varnothing}\\
& \quad + \sum_{I\in S}\sigma(\beta_I)(-a_1^{-1}a_0V_I+a_1^{-1}b_0\alpha^{i_1}V_{I^{\prime}})+
\sum_{I\in S} \delta(\beta_I)V_I\\
& = \sum_{I\in S\cup\{\varnothing\}}\bigl(\delta(\beta_I)-a_1^{-1}a_0\sigma(\beta_I)\bigr)V_I+
\sum_{I\in S}a_1^{-1}b_0\sigma(\beta_I)\alpha^{i_1}V_{I^{\prime}}+
a_1^{-1}b_0\sigma(\beta_{\varnothing}).
\end{split}\end{equation*}
Multiplying this by $a_1$ and summing with $a_0h$, one gets
\begin{equation*}\begin{split}
fh & = (a_0+a_1X)h=a_0h+a_1Xh = \sum_{I\in S\cup\{\varnothing\}}a_0\beta_IV_I + \\
& \quad + \sum_{I\in S\cup\{\varnothing\}}\bigl(a_1\delta(\beta_I)-a_0\sigma(\beta_I)\bigr)V_I+
\sum_{I\in S}b_0\sigma(\beta_I)\alpha^{i_1}V_{I^{\prime}}+
b_0\sigma(\beta_{\varnothing})\\
& = \sum_{I\in S\cup\{\varnothing\}}\psi(\beta_I)V_I + 
\sum_{I\in S}b_0\sigma(\beta_I)\alpha^{i_1}V_{I^{\prime}}+
b_0\sigma(\beta_{\varnothing}).
\end{split}\end{equation*}
Therefore, one has
\begin{equation}\label{eq:5}
\sum_{I\in S\cup\{\varnothing\}}\psi(\beta_I)V_I + 
\sum_{I\in S}b_0\sigma(\beta_I)\alpha^{i_1}V_{I^{\prime}}=
fh-b_0\sigma(\beta_0)\in D[X;\sigma,\delta].
\end{equation}
The coefficient of $V_T$ in \eqref{eq:5} is $\psi(\beta_T)=\psi(1)=0$. Moreover,
no new nonzero coefficient of a $V_I$ with $\length(I)=r$ appears in \eqref{eq:5}.
By the minimality of \eqref{eq:3}, all the coefficients of the $V_I$ in \eqref{eq:5}
are zero. If $\length(I)=r$, the coefficient of $V_I$ in \eqref{eq:5}
is $\psi(\beta_I)$, so, in particular, it follows that 
$\beta_I\in E=\ker\psi$ for all $I\in S$ with
$\length(I)=r$. Now, there are exactly 3 elements $I_0,I_1,I_2$ in $S$ whose truncations
equal $T^{\prime}$. Since all three have length $r$, if follows that $\beta_{I_0},
\beta_{I_1},\beta_{I_2}\in E$. But the coefficient of $V_{T^{\prime}}$ in \eqref{eq:5}
is $\psi(\beta_{T^{\prime}})+b_0\sigma(\beta_{I_0})+b_0\sigma(\beta_{I_1})\alpha
+b_0\sigma(\beta_{I_2})\alpha^2$. So,
$$
\psi(\beta_{T^{\prime}})=\sigma(-b_0\beta_{I_0})+\sigma(-b_0\beta_{I_1})\alpha+
\sigma(-b_0\beta_{I_2})\alpha^2,
$$
which is an element of $\psi(D)\cap\bigl(\sigma(E)+\sigma(E)\alpha+\sigma(E)\alpha^2\bigr) 
= \{0\}$. Since $\{1,\alpha,\alpha^2\}$ is left linearly independent over $\sigma(E)$,
it follows that $\beta_{I_0}=\beta_{I_1}=\beta_{I_2}=0$. But $T\in\{I_0,I_1,I_2\}$. This
contradicts the fact that $\beta_T=1$.

\medskip

\textit{Now suppose that condition \ref{cond:2} holds, that is, that $b_0=0$ and
$\delta=0$.} In this case, we must have $b_1\neq 0$ and $a_0\neq 0$.
We shall show the stronger statement 
that $\{V_I : I\in S\cup\{\varnothing\}\}$
is left linearly independent over $D$ modulo the subspace $D[X,X^{-1};\sigma]$.
By contradiction, suppose there exists a relation
\begin{equation}\label{eq:31}
\sum_{I\in S\cup\{\varnothing\}} \beta_IV_I = h\in D[X,X^{-1};\sigma],
\end{equation}
with $\beta_I\in D$ not all zero. Among all those relations, choose one with 
$r=\max\{\length(I) : \beta_I\neq 0\}$ minimal. Moreover, among those, choose
one with the smallest number of nonzero coefficients $\beta_I$ for $I$ with
$\length(I)=r$. Note that $r\geq 1$, otherwise we would have $\Xi\in D[X,X^{-1};\sigma]$, 
which is impossible (for $a_0\ne 0$). 
Clearly, we can further assume that our relation \eqref{eq:3},
beyond being minimal in the sense described above, has $\beta_T=1$ for some
$T\in S$ with $\length(T)=r$, by multiplying it by a nonzero element of $D$ on
the left if necessary.

It follows from \eqref{eq:vi} that
\begin{equation}\label{eq:41}
X^{-1}V_{\varnothing} = -a_1a_0^{-1}V_{\varnothing}+b_1a_0^{-1}
\quad\text{and}\quad
X^{-1}V_I=-a_1a_0^{-1}V_I+b_1a_0^{-1}\alpha^{i_1}V_{I^{\prime}},
\end{equation}
for all $I\in S$. If one multiplies \eqref{eq:31} by $X^{-1}$ on the left,
relations \eqref{eq:41} allow us to conclude that
$$
X^{-1}h=\sum_{I\in S\cup \{\varnothing\}}-a_1a_0^{-1}\sigma^{-1}(\beta_I)V_I+
\sum_{I\in S}b_1a_0^{-1}\sigma^{-1}(\beta_I)\alpha^{i_1}V_{I^{\prime}}
+b_1a_0^{-1}\sigma^{-1}(\beta_{\varnothing}).
$$
This multiplied by $a_1^{-1}a_0^2$ and, then, summed with $-a_0h$ yields
\begin{multline}\label{eq:51}
\sum_{I\in S\cup\{\varnothing\}}\psi(\sigma^{-1}(\beta_I))V_I
-\sum_{I\in S}b_1a_1^{-1}a_0\sigma^{-1}(\beta_I)\alpha^{i_1}V_{I^{\prime}} \\=
-(a_1^{-1}a_0^2X^{-1}+a_0)h+b_1a_1^{-1}a_0\sigma^{-1}(\beta_{\varnothing})\in
D[X,X^{-1};\sigma].
\end{multline}
The coefficient of $V_T$ in \eqref{eq:51} is $\psi(\sigma^{-1}(\beta_T))=\psi(1)=0$.
By minimality, all the coefficients on the left-hand side of \eqref{eq:51} are zero.
In particular, if $\length(I)=r$, the coefficient of $V_I$ is $0=\psi(\sigma^{-1}(\beta_I))$.
So, for $I$ with $\length(I)=r$, one has $\beta_I\in\sigma(E)=E$. (This last equality
follows from the fact that, in this case, $E=\ker(\Id-\sigma)$; so $\sigma(E)=E$.)
The rest of the argument is analogous to the one in the first case.
\end{proof}

\begin{theorem}\label{th:2}
Let $n$ be an integer with $n\geq 2$.
Let $\alpha_1,\dots,\alpha_n\in D$ be such that
\begin{itemize}
\item $\{\alpha_1,\dots,\alpha_n\}$ is left linearly independent over $\sigma(E)$ and
\item $\psi(D)\cap\bigl(\sigma(E)\alpha_1+\dots+\sigma(E)\alpha_n\bigr) = \{0\}.$
\end{itemize}
If either
\begin{enumerate}[label=(\roman*)]
\item $b_1=0$ or \label{cond:21}
\item $b_0=0$ and $\delta=0$, \label{cond:22}
\end{enumerate}
then the set $\{\alpha_1\Xi,\dots,\alpha_n\Xi\}$ freely generates a free $k$-subalgebra in $D(X;\sigma,\delta)$.
\end{theorem}

\begin{proof} (Sketch.) We consider the set
\begin{multline*}
S=\Bigl\{\bigl((i_1),\dots,(i_t)\bigr) : t\geq 1, (i_j)=(i_{j1},\dots,i_{jn}), i_{jl}\in\{0,1\}, \\
\sum_{l=1}^n i_{jl}=1, \text{ for all $j=1,\dots,t$}\Bigr\}.
\end{multline*}
Given $I=\bigl((i_1),\dots,(i_t)\bigr)\in S$, one defines
$$
W_I=\alpha_1^{i_{11}}\dots\alpha_n^{i_{1n}}\Xi\alpha_1^{i_{21}}\dots\alpha_n^{i_{2n}}\Xi
\dots\alpha_1^{i_{t1}}\dots\alpha_n^{i_{tn}}\Xi.
$$
The set of all nonempty words in the letters $\alpha_1\Xi,\dots,\alpha_n\Xi$ coincides with
$\{W_I : I\in S\}$. Our task is, thus, to show that $\mathcal{B}=\{1\}\cup\{W_I : I\in S\}$ is linearly independent
over $k$.

Here, for $I=\bigl((i_1),\dots,(i_t)\bigr)\in S$, its length is defined to be $t$ and its
truncation $I^{\prime}=\bigl((i_2),\dots,(i_t)\bigr)\in S$, if $t\geq 2$. If $I$ has length
$1$, its truncation is defined to be $I^{\prime}=\varnothing$. It follows from the definition
of $S$ that given $I\in S$, there exist exactly $n$ elements of $S$, all of them with
the same length as $I$, having truncation $I^{\prime}$ (clearly, one of them is $I$
itself).

Defining $V_I=\Xi W_I$, 
for $I\in S$, and $V_{\varnothing}=\Xi$, one can show, following the lines of the proof of 
Theorem~\ref{th:1}, above, that, first, if $\bigr\{V_I : I\in S\cup\{\varnothing\}\bigl\}$ is 
left linearly independent over $D$, then $\mathcal{B}$ is linearly
independent over $k$. Moreover, the proof, in Theorem~\ref{th:1}, that $\bigr\{V_I : I\in S\cup\{\varnothing\}\bigl\}$ is left linearly independent over $D$, under both condition
\ref{cond:21} or condition \ref{cond:22}, can also be adapted to the present context.
\end{proof}

\begin{remark}\label{rem:weyl}
Setting $\sigma$ to be the identity automorphism of $D$,
Theorem~\ref{th:2} can be used to recover both \cite[Theorem A]{SG98} and
Makar-Limanov's result of \cite{ML83}, producing free subalgebras inside the division
ring of fractions of the first Weyl algebra over the rationals. Indeed, if 
$D_1$ denotes the division ring of fractions of the first Weyl algebra
$A_1=\Q\langle s,t : st-ts=1\rangle$, then, via the identification $s\mapsto X$, 
$D_1$ coincides with
the division ring of fractions $\Q(t)(X;\delta)$ of the skew polynomial
ring $\Q(t)[X;\delta]$, where $\delta$ is the usual derivation on the rational
function field $\Q(t)$, that is, the one satisfying $\delta(t)=1$. Here,
the rational functions $\alpha_1=\frac{1}{t}$ and $\alpha_2=\frac{1}{t(1-t)}$
satisfy the hypotheses of Theorem~\ref{th:2}; hence, taking $a_0=b_0=0$ and $a_1=b_1=1$,
it follows that $\alpha_1X^{-1}$ and $\alpha_2X^{-1}$ generate a free
$\Q$-subalgebra in $\Q(t)(X;\delta)$, or, in other words, $(st)^{-1}$ and $(1-t)^{-1}(st)^{-1}$
generate a free $\Q$-subalgebra of $D_1$.

Observe that Theorem~\ref{th:2} recovers Makar-Limanov's result, which does not occur with 
Theorem~2.2 in \cite{BR12}, as pointed out by Bell and Rogalski.

In Section~\ref{sec:weyl}, we shall see that Theorem~\ref{th:2} can also provide a
pair of \emph{symmetric} elements of $D_1$ generating a free algebra, with respect
to a natural involution on $D_1$.
\end{remark}

\section{Free symmetric subalgebras and the Heisenberg group}

Let $k$ be a field, let $\Gamma=\langle{x,y : [[x,y],x]=[[x,y],y]=1}\rangle$
be the Heisenberg group and let $\ast$ be an involution on $\Gamma$. Then
$\ast$ can be linearly extended to a $k$-involution $\ast$ on the
group algebra $k\Gamma$, which, in turn, has a unique extension to
a $k$-involution on the Ore division ring of fractions $D$ of the
noetherian domain $k\Gamma$.

In this section, we shall present a proof of Theorem~\ref{th:nilpgp}, exhibiting
two elements in $D$ which freely generate a free $k$-subalgebra
and which are symmetric with respect to $\ast$. For that purpose, we shall make
use of Theorem~\ref{th:1} and of the classification of involutions on $\Gamma$
given in \cite{FG15}.

Recall that the center of $\Gamma$ 
is infinite cyclic, generated by $\lambda = [x,y]$.
The attribution $\lambda\mapsto t, y\mapsto Y, x\mapsto X$ establishes
a $k$-isomorphism between $D$ and the division ring $k\bigl((t)(Y)\bigr)(X;\sigma)$, where
$k(t)$ stands for the field of rational functions in the indeterminate $t$ over $k$,
$k(t)(Y)$ for the field of rational functions in the indeterminate $Y$ over $k(t)$,
and $\sigma$ is the $k(t)$-automorphism of $k(t)(Y)$ satisfying $\sigma(Y)=tY$.

Theorem~\ref{th:nilpgp} will follow from Theorem~\ref{th:1}, after a judicious
choice of elements $\alpha$ and $\Xi$. 
But, in order to verify the hypotheses of Theorem~\ref{th:1} in this setting,
we shall need the following fact on automorphisms of rational function fields,
whose proof is similar to the proof of \cite[Lemma~1.4]{FGS13}.

\begin{lemma}\label{le:rational}
Let $F$ be a field, let $t\in F\setminus\{0\}$ be an element which is not
a root of unity, and let $\sigma$ be the $F$-automorphism of the rational function
field $F(Y)$ such that $\sigma(Y)=tY$. Let $\alpha\in F(Y)\setminus F[Y]$ be a
rational function which has a unique pole and this pole is nonzero, and let $m$ 
be a positive integer. 
If $\beta\in F(Y)$ satisfies 
$$
\sigma(\beta)-\beta\in F+F\alpha+\dots+F\alpha^m,
$$
then $\beta\in F$. \qed
\end{lemma}

\subsection{Proof of Theorem~\ref{th:nilpgp}}
As we have seen above, we can identify $D$ with $\bigl(k(t)(Y)\bigr)(X;\sigma)$. 
Taking $F=k(t)$ in Lemma~\ref{le:rational}, 
one sees that any rational function $\alpha\in F(Y)$ which has a unique 
pole and this pole is nonzero will satisfy the hypotheses of Theorem~\ref{th:1}, 
therefore providing a pair
$\{\alpha X(1-X)^{-1}, X(1-X)^{-1}\alpha\}$ inside $D$ which freely generates
a free $k$-subalgebra. Now, 
according to \cite[Theorem~3.4]{FG15}, up to equivalence,
a $k$-involution $\ast$ on $D$ which is induced by an involution on $\Gamma$ must satisfy one
of the following conditions:
\begin{enumerate}[label=(\Roman*)]
\item $X^{\ast}=\zeta X, \ Y^{\ast}=\eta Y$;\label{i1}
\item $X^{\ast}=X^{-1}, \ Y^{\ast}=Y^{-1}$;\label{i2}
\item $X^{\ast}=X, \ Y^{\ast}=\zeta Y^{-1}$;\label{i3}
\item $X^{\ast}=\zeta Y, \ Y^{\ast}=\zeta^{-1}X$;\label{i4}
\end{enumerate}
the elements $\zeta$ and $\eta$ being powers of $t$ (and, therefore, central).
In the first two cases, one has $t^{\ast}=t^{-1}$, and in the last two,
$t$ is symmetric.

We shall treat each of the four types \ref{i1}-\ref{i4} separately.
\begin{enumerate}[label=(\Roman*)]
\item In this case, taking $\alpha=(1-Y)^{-1}$, we obtain elements
$A=(1-Y)^{-1}X(1-X)^{-1}$ and $B=X(1-X)^{-1}(1-Y)^{-1}$ freely generating
a free subalgebra of $D$. 
Now consider the $k(t)$-automorphism $\psi$ of $D$ such that $\psi(Y)=(1+\eta)Y$
and $\psi(X)=(1+\zeta)X$. Since $(1+\eta)Y=Y+Y^{\ast}$ and $(1+\zeta)X=X+X^{\ast}$, it
follows that $\psi(Y)$ and $\psi(X)$ are symmetric with respect to $\ast$. Thus,
$\psi(A)^{\ast}=\psi(B)$. This implies that $\psi(AB)$ and $\psi(BA)$ are symmetric
and, because $AB$ and $BA$ freely generate a free subalgebra of $D$, so do they.
\item This is contained in Theorem~1.1 of \cite{FGS13}.
\item The rational function $\gamma=Z(\zeta-Z)^{-2}$ in the indeterminate $Z$ over
the field $F=k(t)$ satisfies the conditions of Lemma~\ref{le:rational} with respect
to the automorphism $\tau$ such that $\tau(Z)=t^2Z$. Therefore, by Theorem~\ref{th:1},
$\gamma X(1-X)^{-1}$ and $X(1-X)^{-1}\gamma$ freely generate a free $k$-subalgebra
in $\bigl(k(t)(Z)\bigr)(X;\tau)$. Since the map $Z\mapsto Y^2$ establishes an isomorphism
between $\bigl(k(t)(Z)\bigr)(X;\tau)$ and the subalgebra $\bigl(k(t)(Y^2)\bigr)(X;\sigma)$
of $D$, it follows that, setting $\alpha=Y^2(\zeta-Y^2)^{-2}$, the elements
$A=Y^2(\zeta-Y^2)^{-2}X(1-X)^{-1}$ and $B=X(1-X)^{-1}Y^2(\zeta-Y^2)^{-2}$ freely generate
a free $k$-subalgebra of $D$. Since $A^{\ast}=B$,
it follows that $AB$ and $BA$ form a pair of symmetric elements which freely
generate a free subalgebra of $D$.
\item Here, taking $\alpha=Y(1-Y)^{-1}$, one gets the free pair $A=Y(1-Y)^{-1}X(1-X)^{-1}$
and $B=X(1-X)^{-1}Y(1-Y)^{-1}$. If $\psi$ denotes the $k(t)$-automorphism of
$D$ such that $\psi(X)=X$ and $\psi(Y)=\zeta Y$, it follows that 
$\{\psi(A),\psi(B)\}$ is a pair of symmetric elements which freely generates 
a free algebra in $D$.
\end{enumerate}

\section{Free symmetric subalgebras and the first Weyl algebra}\label{sec:weyl}

As we have seen in Remark~\ref{rem:weyl}, we can regard the division ring
of fractions $D_1$ of the first Weyl algebra over $\Q$ as $\Q(t)(X;\delta)$, where 
$\delta$ stands for the usual derivation on the rational
function field $\Q(t)$.

In the proof of Theorem~\ref{th:weyl}, we shall need the following consequence
of Theorem~\ref{th:2}.

\begin{corollary}\label{cor:th2}
Let $a,b\in \Q(t)$ be rational functions satisfying the following conditions:
\begin{itemize}
\item $\{a^2,ab\}$ is a $\Q$-linearly independent subset of $\Q(t)$, and
\item $\delta\bigl(\Q(t)\bigr)\cap (\Q a^2+\Q ab) = \{0\}$.
\end{itemize}
Then, $aX^{-1}a$ and $bX^{-1}a$ freely generate a free $\Q$-subalgebra
of $\Q(t)(X;\delta)$.
\end{corollary}

\begin{proof}
By Theorem~\ref{th:2}, the elements $a^2X^{-1}$ and $abX^{-1}$ freely generate
a free $\Q$-subalgebra of $\Q(t)(X;\delta)$. Now, consider the set of monomials
on the letters $A$ and $B$, and given $I=(i_1,j_1,i_2,j_2,\dots,i_n,j_n)$ with
$i_k, j_k$ nonnegative integers, let $M_I(A,B)$ be the monomial defined by
$$
M_I(A,B) = A^{i_1}B^{j_1}A^{i_2}B^{j_2}\dots A^{i_n}B^{j_n}.
$$
Then, for any $I$, we have that 
\begin{equation}\label{eq:freeset}
aM_I(aX^{-1}a,bX^{-1}a)X^{-1}= M_I(a^2X^{-1},abX^{-1})aX^{-1}.
\end{equation}
Hence, if $c_I\in \Q$ are such that only a finite number of them are
nonzero and $\sum_I c_IM_I(aX^{-1}a,bX^{-1}a)=0$, multiplying this relation
by $a$ on the left and by $X^{-1}$ on the right, we get, using \eqref{eq:freeset},
$$
0=a\left(\sum_I c_IM_I(aX^{-1}a,bX^{-1}a)\right) X^{-1}=
\left(\sum_I c_I M_I(a^2X^{-1},abX^{-1})\right)aX^{-1}.
$$
Since the set $\{a^2X^{-1},abX^{-1}\}$ is free, it follows that all the $c_I$
are zero. Therefore, $\{aX^{-1}a, bX^{-1}a\}$ is also free.
\end{proof}

\subsection{Proof of Theorem~\ref{th:weyl}}
Consider the rational functions
$$
a=\frac{t}{1+t^2} \qquad\text{and}\qquad b=\frac{1}{1+t}
$$
in $\Q(t)$. Considering them as real functions in the variable $t$, we have
$$
\int \left(\frac{t}{1+t^2}\right)^2 dt = \frac{1}{2}\left(\arctan t-\frac{t}{1+t^2}\right) 
+ \mathrm{constant}
$$ 
and
$$
\int \left(\frac{t}{1+t^2}\right)\left(\frac{1}{1+t}\right) dt = \frac{1}{4}\left(\ln(1+t^2)
+ 2\arctan t - 2\ln(1+t) \right) + \mathrm{constant}.
$$
Developing $\arctan t$, $\ln(1+t^2)$ and $\ln(1+t)$ as power series in the interval
$(0,1)$, we can easily check that $a$ and $b$ satisfy the conditions in
Corollary~\ref{cor:th2}. It follows that $\alpha=as^{-1}a$ and $\beta=bs^{-1}a$ freely
generate a free $\Q$-subalgebra of $D_1$. Hence, the symmetric elements $\alpha^2$ and 
$\alpha\beta$ also generate a free $\Q$-subalgebra of $D_1$.

\section{Free subalgebras in $k(X_1,\dots,X_n)(X;\sigma)$}

In this section we follow closely the arguments in \cite[Section~4]{SG99} and
show that part of the proof of \cite[Theorem~1]{SG99} can be greatly
simplified using Theorem~\ref{th:2}.

We start with a more general setting. Let $k$ be a field and let $R$ be a commutative $k$-algebra which is a factorial
domain with group of units $k^{\dagger}=k\setminus\{0\}$. Let $\sigma$ be a nonidentity
$k$-automorphism of $R$ and assume the the fixed ring of $R$ under $\sigma$ coincides
with $k$. Extend $\sigma$ to the field of fractions $K$ of $R$. Theorem~\ref{th:ratfunct}
will follow from the next result, in the statement of which, for $a\in k^{\dagger}$, we
use the notation $R_a =\{r\in R : \sigma(r)=ar\}$.

\begin{proposition}\label{prop:factorial}
Under the above hypotheses, the division algebra $K(X;\sigma)$ contains a noncommutative
free $k$-subalgebra. More precisely, one of the following alternative possibilities
must hold.
\begin{enumerate}[label=(\roman*)]
\item Either $R_a=\{0\}$, for all $a\in k^{\dagger}\setminus\{1\}$. In this case, given any
$\alpha\in K\setminus R$ whose denominator is a prime power, for any positive
integer $m$, the set
$$\{\alpha X(1-X)^{-1}, \alpha^2 X(1-X)^{-1},\dots, \alpha^m X(1-X)^{-1}\}$$
freely generates a free $k$-subalgebra in $K(X;\sigma)$.\label{cond:I}
\item Or $R\supseteq k[t]$, where $t$ is algebraically independent over $k$ and $\sigma$
satisfies $\sigma(t)=\lambda t$, for some $\lambda\in k$ which is not a root of
unity. In this case, given any $b\in k$, for any positive integer $m$, the set
$$\{(t-b)^{-1}X(1-X)^{-1}, (t-b)^{-2}X(1-X)^{-1},\dots,(t-b)^{-m}X(1-X)^{-1}\}$$
freely generates a free $k$-subalgebra in $K(X;\sigma)$.\label{cond:II}
\end{enumerate}
\end{proposition}

\begin{proof}
In case \ref{cond:I}, take $\alpha\in K\setminus R$. By \cite[Lemma~5]{SG99}, the set
$\{1\}\cup\{\sigma^j(\alpha^i) : i\geq 1, j\geq 0\}$ is $k$-linearly independent. Moreover,
if the denominator of $\alpha$ is a prime power, then, by \cite[Lemma~7]{SG99},
the equation
$$\sigma(\beta)-\beta = \sum_{i\geq 1} b_i\alpha^i$$
has no solution with $b_i\in k$ and $\beta\in K\setminus k$. It follows from 
Theorem~\ref{th:2} that $\alpha X(1-X)^{-1},
\dots, \alpha^m X(1-X)^{-1}$ freely generate a free $k$-algebra in $K(X;\sigma)$ for
any positive $m$.

Now suppose that \ref{cond:I} does not hold, that is, there exists $\lambda\in 
k^{\dagger}\setminus\{1\}$ such that $R_{\lambda}\ne \{0\}$. By \cite[Lemma~2]{SG99}, 
$\lambda$ is not a root of unity. Choose $t\in R_{\lambda}\setminus\{0\}$. Then, 
$\sigma(t)=\lambda t$ and we have an embedding $k(t)(X;\sigma)\subseteq K(X;\sigma)$.
It follows from Lemma~\ref{le:rational} and Theorem~\ref{th:2} that, for any $b\in k$ and any
positive integer $m$, $(t-b)^{-1}X(1-X)^{-1}, (t-b)^{-2}X(1-X)^{-1},\dots,
(t-b)^{-m}X(1-X)^{-1}$ freely generate a free $k$-subalgebra in $k(t)(X;\sigma)$ and,
hence, in $K(X;\sigma)$.
\end{proof}

\subsection{Proof of Theorem~\ref{th:ratfunct}}
The same argument used in the proof of \cite[Corollary~2]{SG99} holds. Let $M$
be the fixed subring of $S=k[X_1,\dots, X_n]$ under the action of $\sigma$,
let $R=S(M\setminus\{0\})^{-1}$, and let $k=M(M\setminus\{0\})^{-1}$. By 
Proposition~\ref{prop:factorial}, $K(X;\sigma)$ contain a free $k$-subalgebra and,
thus, by \cite[Lemma~1]{MM91}, contains a free $F$-subalgebra.

\end{document}